\documentclass{amsart}

\usepackage[english]{babel}

\usepackage[letterpaper,top=2cm,bottom=2cm,left=3cm,right=3cm]{geometry}

\usepackage[style=alphabetic,doi=false,isbn=false,url=false,eprint=true,
backend=biber,giveninits=true,maxcitenames=3,maxbibnames=8,hyperref]{biblatex}
\usepackage{amsmath}
\usepackage{amsfonts}
\usepackage{amssymb}
\usepackage{amsthm}
\usepackage{graphicx}
\usepackage[colorlinks=true, allcolors=blue]{hyperref}

\usepackage{xcolor}
\usepackage{subcaption}
\usepackage{todonotes}
\usepackage{cleveref}

\numberwithin{equation}{section}

\theoremstyle{plain}
\newtheorem{lemma}[equation]{Lemma}
\newtheorem{theorem}[equation]{Theorem}

\newtheorem{corollary}[equation]{Corollary}
\newtheorem{proposition}[equation]{Proposition}
\theoremstyle{definition}
\newtheorem{definition}[equation]{Definition}
\newtheorem{remark}[equation]{Remark}

\newtheorem{example}[equation]{Example}

\newcommand{\M}{\mathcal{M}}
\newcommand{\cL}{\mathcal{L}}
\newcommand{\C}{\mathcal{C}}
\newcommand{\B}{\mathcal{B}}

\def\Z{{\mathbb{Z}}}

\def\K{{\mathbb{K}}}
\def\F{{\mathbb{F}}}

\title{Most $q$-matroids are not representable}

\author{Sebastian Degen}
\address{Universit\"at Bielefeld, Fakult\"at f\"ur Mathematik, Bielefeld, Germany}
\email{sdegen@math.uni-bielefeld.de}

\author{Lukas~K\"uhne}

\email{lkuehne@math.uni-bielefeld.de}

\begin{document}
\maketitle

\begin{abstract}
A $q$-matroid is the analogue of a matroid which arises by replacing the finite ground set of a matroid with a finite-dimensional vector space over a finite field.
These $q$-matroids are motivated by coding theory as the representable $q$-matroids are the ones that stem from rank-metric codes.
In this note, we establish a $q$-analogue of Nelson's theorem in matroid theory by proving that asymptotically almost all $q$-matroids are not representable.
This answers a question about representable $q$-matroids by Jurrius and Pellikaan strongly in the negative.
\end{abstract}

\hfill{{\em In memory of Kai-Uwe Schmidt.}}
\section{Introduction}

The concept of $q$-matroids as $q$-analogues of matroids  by replacing finite sets with finite dimensional vector spaces over finite fields was first described in Crapo's Ph.D. thesis~\cite{crapo}.
They were recently re-introduced by Jurrius and Pellikaan in \cite{jp2018} with a focus on their connections to coding theory.

We start by recalling the definition of $q$-matroids. In the following let $E$ be a finite dimensional vector space over a finite field $\F_q$, for some prime power $q$ and denote by $\cL(E)$ its lattice of subspaces. 

\begin{definition}
    A \textbf{$q$-matroid} $\M$ is a pair $(E,\rho)$ of a vector space $E$ as above and a function $\rho:\cL(E)\rightarrow\Z_{\geq 0}$, called the \textbf{$q$-rank function} satisfying the following axioms for all subspaces $X,Y\leq E$:
    \begin{enumerate}
        \item  $0\leq \rho(X)\leq \dim(X)$,
        \item  If $X\leq Y$, then $\rho(X)\leq \rho(Y)$ and
        \item  $\rho(X\cap Y)+\rho(X+Y)\leq \rho(X)+\rho(Y)$.
    \end{enumerate}
\end{definition}
These axioms are a direct translation of the usual axioms of a matroid in terms of its rank function to the setting of a vector space and its subspaces.

One of the main motivations to study $q$-matroids stems from coding theory, as the representable $q$-matroids arise from rank-metric codes. Recently, several new phenomena about rank-metric codes were discovered by studying their underlying $q$-(poly)matroids, see for instance \cite{gjlr2019,gj2022,gj2023,ab2023_1}. 
On the other hand, studying $q$-matroids as analogues of matroids is also an active field of research.~Significant effort has been invested to define $q$-analogue versions of matroidal concepts, such as $q$-cryptomorphic axioms systems, see \cite{ab2023_2,cj2023,bcj2022}, or a $q$-analogue of the direct sum of matroids, see \cite{cj2022}.
It is however fair to say that in these cases the $q$-analogue is much more involved than the immediate translation of the matroid rank function above.

Jurrius and Pellikaan asked whether all $q$-matroids are representable~\cite{jp2018}.
This question turned out to be too optimistic, as the first examples of non-representable appeared recently: 
Gluesing-Luerssen and Jany introduced a method of translating well-known non-representable matroid such as the V\'amos matroid to the $q$-analogue setting~\cite{gj2022} whereas Ceria and Jurrius found the smallest non-representable $q$-matroid which is of rank $2$ on $\F_2^4$~\cite{cj2022}.

The main result of this paper answers the above question negatively in a strong sense by proving the following asymptotic result on representable $q$-matroids.

\begin{theorem}\label{q_mat_asymptotics}
    Let $n$ be an integer, $\mathcal{R}_{q}(n)$ the number of representable $q$-matroids and $\mathcal{N}_{q}(n)$ the number of all $q$-matroids on $\F_q^n$, respectively. Then the ratio $\frac{\mathcal{R}_{q}(n)}{\mathcal{N}_{q}(n)}$ asymptotically vanishes, i.e.,
    \[
    \lim_{n\rightarrow\infty}\frac{\mathcal{R}_{q}(n)}{\mathcal{N}_{q}(n)}=0.
    \]
\end{theorem}

This result implies that the portion of representable $q$-matroids tends to $0$ as $n$ goes to infinity. 
This is a $q$-analogue version of a celebrated theorem of Nelson which says that asymptotically almost all matroids are not representable~\cite{n2018}.
To prove this theorem, we provide a lower bound on the number $\mathcal{N}_{q}(n)$ via coding theoretic estimations on constant dimension codes and an upper bound on the number $\mathcal{R}_{q}(n)$ via the algebraic concept of zero patterns.

Our paper is structured as follows. In \Cref{preliminaries} we briefly recap the basic notions of $q$-matroids and rank metric codes, needed for the discussion in the later sections. In \Cref{lower_bound} we explain the concept of constant dimension codes and describe a lower bound on their maximal cardinality. Afterwards, we establish their connection to $q$-matroids which yields the lower bound on the number $\mathcal{N}_{q}(n)$. In \Cref{upper_bound} we define the notion of zero patterns of a polynomial sequence and explain how these are connected to representable $q$-matroids. Subsequently, we derive an upper bound on the number of representable $q$-matroids of rank $k$ in $\F_q^n$. We conclude this section with a discussion about the asymptotic behavior of representable $q$-matroids and the proof of our main result, \Cref{q_mat_asymptotics}.

\section*{Acknowledgments}

We thank Gianira Nicoletta Alfarano and Charlene Weiß for fruitful discussions.
We are grateful to Kai-Uwe Schmidt for introducing us to the topic of $q$-matroids.
Lastly, we would like to thank the anonymous referees for their helpful comments on an earlier version of the article.

Both authors are supported by the Deutsche Forschungsgemeinschaft (DFG, German Research Foundation) -- SFB-TRR 358/1 2023 -- 491392403 and SPP 2458 -- 539866293.

\section{Preliminaries}
\label{preliminaries}

\subsection{Notation} Throughout the paper we always think of $E$ as $\F_q^n$ for some prime power $q$ and some integer $n\geq 1$. Moreover, we denote by ${E\choose k}_q$ the set of all $k$-dimensional subspaces of $E$. Finally we abbreviate the row space of a matrix $M\in\F_q^{k\times n}$ as $\text{rowspan}(M)\in\cL(\F_q^n)$.

\subsection{$q$-Binomial coefficients}

 The $k$-dimensional subspaces in $\F_q^n$ for $0\leq k\leq n$, are counted via the \textbf{$q$-binomial coefficient}
\[
    {n\choose k}_q:=\frac{(q^n-1)\cdots(q^{n-k+1}-1)}{(q-1)\cdots(q^k-1)}.
\]

The following bounds on the $q$-binomial coefficient will be crucial for the estimations later on in this paper, see \cite[Lemma 2.1, Lemma 2.2]{i2015}.

\begin{lemma}\label{Estimations_q_binom}
    For $0\leq k\leq n$ and $q\geq 2$, the following holds
        \[
        q^{(n-k)k}\leq{n\choose k}_q\leq \frac{111}{32}q^{(n-k)k}.
        \]
\end{lemma}

\subsection{$q$-Matroids}
\label{q_mats}

In this subsection, we shortly recall the basic notions about $q$-matroids, which we frequently use throughout the article. For more details see for instance \cite{jp2018,bcj2022}.

In the introduction, we recalled the definition of a $q$-matroid.~Given such a $q$-matroid $\M=(E,\rho)$ we are interested in its associated concepts.
We call $\rho(\M):=\rho(E)$ the \textbf{rank} of $\M$. A subspace $X\in\cL(E)$ is called \textbf{independent} if $\rho(X)=\dim X$, otherwise it is called \textbf{dependent}. If an independent space~$B$ satisfies $\rho(B)=\rho(E)$, it is called a \textbf{basis} of $\M$. Finally, we call a subspace $C$ a \textbf{circuit} if it is dependent and all its proper subspaces are independent.

As with usual matroids, there are cryptomorphic characterizations of $q$-matroids by its collection of independent spaces, dependent spaces, bases and circuits, see \cite{jp2018,bcj2022, cj2023}.
Somewhat surprisingly unless in the case of the characterization by rank function, these alternative descriptions are generally more involved than the direct translations of the usual matroidal axioms.

In \Cref{lower_bound} we need a special construction of so-called \textbf{paving} $q$-matroids, which are $q$-matroids~$\M$ such that every circuit $C$ in $\M$ satisfies $\dim(C)\geq \rho(\M)$.
One way to obtain paving $q$-matroids is via the following construction, described in~\cite{gj2022}.

\begin{proposition}\cite[Proposition 4.5]{gj2022}
    Let $n$ be an integer and fix $1\leq k\leq n-1$. Further let $\mathcal{S}$ be a collection of $k$-dimensional subspaces of $\F_q^n$ such that for every two distinct $V,W\in\mathcal{S}$, $\dim(V\cap W)\leq k-2$. Define the map
    
    $$\rho:\cL(E)\rightarrow\mathbb{Z}_{\geq 0},\quad V\mapsto
    \begin{cases}
        \;\hfil k-1\;&\emph{if }V\in\mathcal{S},\\
        \;\hfil \min\{\dim V,k\}\;&\emph{otherwise}. 
    \end{cases}$$
    Then $(E,\rho)$ is a paving $q$-matroid of rank $k$, whose circuits of rank $k-1$ are the subspaces in $\mathcal{S}$.
    \label{paving_construction}
\end{proposition}

Another crucial concept stemming from matroid theory is the duality of $q$-matroids.
This is again inspired by the usual matroid duality.
\begin{proposition}\cite[Theorem 42]{jp2018}
    Let $\M=(E,\rho)$ be a $q$-matroid and $\B$ its collection of bases. Let $\perp$ denote the orthogonal complement w.r.t.\ a non-degenerated bilinear form on $E$.
    Define the function $\rho^*$ by setting
    \begin{align*}
    	&\rho^*:\cL(E)\rightarrow\Z_{\geq 0}, \quad    	\rho^*(V)=\dim(V)+\rho(V^\perp)-\rho(E).
    \end{align*}
    
    Then $\M^*=(E,\rho^*)$ is a $q$-matroid, called the \textbf{dual $q$-matroid} of $\M$. Moreover, the collection of bases $\B^*$ of $\M^*$ are the orthogonal complements of the elements in $\B$.  
\end{proposition}
By definition, the dual rank is given by $\rho^*(\M^*)=n-\rho(\M)$ and we naturally have $(\M^*)^*=\M$. We want to emphasize here that there exists a bijection between $q$-matroids in $\F_q^n$ of rank $k$ and those of rank $(n-k)$ given by the map which sends a $q$-matroid to its dual, see \cite[Section 8]{jp2018}. Therefore we may restrict ourselves to considering $q$-matroids of rank $k$ with  $0\leq k\leq\lfloor\frac{n}{2}\rfloor$ where $n=\dim E$.

We conclude this subsection with two examples.

\begin{example}\label{exp:uniform}
	Set $E=\F_q^n$ for some prime power $q$ and integer $n\geq 1$.
    For a number $0\leq k\leq n$, we define the \textbf{uniform $q$-matroid} $\mathcal{U}_{k,n}(E):=(E,\rho)$ of rank $k$ and dimension $n$, where $\rho$ is given by
    \[
    	\rho(V)=\min\{k,\dim V\},\;\text{for all }V\in\cL(E).
    \]
    So the collection of bases consists of all ${n\choose k}_q$ many $k$-dimensional subspaces of $E$.
    Notice that  $\mathcal{U}_{k,n}(E)$ is a paving $q$-matroid, since the collection of circuits consist of all ${n\choose k+1}_q$ many $(k+1)$-spaces of $E$.
    Moreover, $\mathcal{U}_{k,n}(E)$ arises from the construction in Proposition \ref{paving_construction}, by setting $\mathcal{S}=\emptyset$.
    
    Finally the dual $\mathcal{U}_{k,n}^*(E)$ is the uniform $q$-matroid $\mathcal{U}_{n-k,n}(E)$.
\end{example}

\begin{example}
    Let $E=\F_2^4$, $k=2$ and $\mathcal{S}$ be the collection of the $2$-dimensional subspaces $V,W$ generated by the rows of these two matrices $A$ and $B$, respectively:
    \[
    A=\begin{pmatrix}
        1&0&0&0\\
        0&1&0&0\\
    \end{pmatrix},\quad
    B=\begin{pmatrix}
        0&0&1&0\\
        0&0&0&1\\
    \end{pmatrix}.
    \]

    We clearly have $\dim(V\cap W)=0$, and so the collection $\mathcal{S}$ satisfies the assumptions of \Cref{paving_construction}.
    Thus this yields a paving $q$-matroid of rank $2$ in dimension $4$, having all ${4\choose2}_2=35$ many $2$-spaces as bases, except the two in $\mathcal{S}$.
    \label{q_mat_example}
\end{example}

\subsection{Rank-metric codes and representable $q$-matroids}
\label{rank_metric_code}

This subsection serves as a brief introduction to algebraic coding theory, more specifically we recall the basics about rank-metric codes and their connection to $q$-matroids, see~\cite{ab2023_2,ab2023_1} for further details.

We start with rank-metric codes. For this purpose we endow the vector space $\F_{q^m}^n$ with the so called \textbf{rank distance metric}, defined as $d_{\text{rk}}(v,w):=\text{rk}(v-w)$ for every $v,w\in\F_{q^m}^n$, where for $v=(v_1,\ldots,v_n)\in\F_{q^m}^n$ we set $\text{rk}(v):=\dim_{\F_q}\langle v_1,\ldots,v_n\rangle_{\F_q}$.

\begin{definition}
    We call an $\F_{q^m}$-linear subspace $\C\leq\F_{q^m}^n$ a \textbf{rank-metric code} in $\F_{q^m}^n$. Its \textbf{minimal rank distance} is
    $$d_{\text{rk}}(\C):=\min\{\text{rk}(v)\;|\;v\in\C,v\not=0\}.$$
    
    Moreover if $\C$ has dimension $k$, we call a matrix $G\in\F_{q^m}^{k\times n}$ whose rows generate $\C$, a \textbf{generator matrix} of $\C$.
    Finally we denote by $\C^\perp$ the \textbf{dual code} of $\C$, which is the orthogonal complement of $\C$ with respect to the standard dot product given by $v\cdot w=\sum_{i=1}^n v_i w_i$ for all $v,w\in\F_{q^m}^n$.
\end{definition}

In the literature, sometimes the above definition of a rank-metric code refers to a \textbf{vector rank-metric code}, to distinguish it from the so-called \textit{linear matrix rank-metric code}. 
A \textbf{matrix rank-metric code} is a subset $\C\subseteq\F_q^{n\times m}$, where $\F_q^{n\times m}$ is endowed with the metric $d(A,B):=\text{rk}(A-B)$ for all $A,B\in\F_q^{n\times m}$. It has \textbf{minimal rank-distance} given by $$d(\C):=\min\{\text{rk}(M-N)\;|\;M,N\in\C,M\not= N\}.$$
If $\C$ is not only a subset of $\F_q^{n\times m}$ but also a $\F_q$-linear subspace, it is called \textbf{linear}.

These vector and linear matrix rank-metric codes are related to each other, see \cite{gj2022,ab2023_1} for more details.
In our paper, the term rank-metric code will always refer to a vector rank-metric code, unless otherwise specified.
 
In analogy with usual matroids, we now define representable $q$-matroids as those which arise from rank-metric codes.

\begin{proposition}\cite[Theorem 24]{jp2018}\label{rank_func_rep_q_mats}
    Let $\F_{q^m}$ be a field extension of $\F_q$ and let $G$ be a $k\times n$-matrix with entries in $\F_{q^m}$ where $1\le k \le n$.
    Assume that $G$ has rank $k$.
    Define a map $\rho:\cL(\F_q^n)\rightarrow\Z_{\geq 0}$ via
    $$\rho(V)=\text{rk}_{\F_{q^m}}(GY^T),$$
    where $Y$ is a matrix such that the subspace $V$ is the row space of $Y$. Then $\M_G:=(\F_q^n,\rho)$ is a $q$-matroid of rank $k$, called the \textbf{$q$-matroid represented by $G$}.
\end{proposition}

Note that we can view $G$ as the generator matrix of a rank-metric code $\C\leq\F_{q^m}^n$, therefore $\M_G$ is in the literature also sometimes named the \textbf{$q$-matroid associated to $\C$} and denoted by $\M_\C$, see for instance \cite{ab2023_2}.

\begin{definition}
    A $q$-matroid $\M$ is \textbf{representable} if there exist $m\geq 1$ and a rank-metric code $\C\leq\F_{q^m}^n$, such that the associated $q$-matroid $\M_\C$ equals $\M$. 
\end{definition}

These $q$-matroid representations behave well under duality:

\begin{proposition}\cite[Theorem 48]{jp2018}\label{Duality_rep_q_mats}
    Let $\M_\C$ be a representable $q$-matroid, with associated rank-metric code $\C\leq\F_{q^m}^n$. Then $\M_\C^*=\M_{\C^\perp}$. 
\end{proposition}

Let us illustrate the above theory with an example.

\begin{example}
    Let $\M_G=(\F_2^4,\B)$ be the $q$-matroid of rank $2$ in dimension $4$ associated to the matrix
    $$\begin{pmatrix}
     1&\alpha&0&0\\0&0&1&\alpha^2
     \end{pmatrix}\in\F_{2^4}^4,$$
     where $\alpha$ is the primitive element of the field extension $\F_{2^4}$ over $\F_2$. Then one can compute that the collection of bases of $\M_G$ is given by
     $$\B={\F_2^4\choose 2}_2\setminus\Bigg\{\begin{pmatrix}1&0&0&0\\0&1&0&0\end{pmatrix},\begin{pmatrix}0&0&1&0\\0&0&0&1\end{pmatrix}\Bigg\},$$
     which coincides with the $q$-matroid from \Cref{q_mat_example}.
     Therefore, this $q$-matroid is representable.
\end{example}

We end this subsection by describing a special kind of matrix rank-metric code, which we revisit in \Cref{lower_bound}.
The following result plays a key role in the definition of this special class of codes and was first proved by Delsarte~\cite{d1978}.

\begin{proposition}[Singleton-like Bound]\label{Singelton_blound}
    Let $\C\leq\F_q^{n\times m}$ be a $k$-dimensional matrix rank-metric code with minimal rank distance $d:=d(\C)$. Then we have
    \[
        k\leq\max(m,n)(\min(m,n)-d+1).
	\]
\end{proposition}

Since vector rank-metric and linear matrix rank-metric codes are related, we can also state the above singleton-like bound for the former codes, see \cite[Remark 3.6.]{g2019}. More precisely if $\C\leq\F_{q^m}^{n}$ is a $k$-dimensional vector rank-metric code with minimal rank distance $d:=d_{\text{rk}}(\C)$, it must then hold that
\[
        k\leq n-d+1.
\]
Matrix or vector rank-metric codes attaining the above bounds are called \textbf{maximal (matrix/vector) rank-metric codes (MRD codes)}.  Their connection to $q$-matroids is stated in the next example.

\begin{example}
    Let $0<k<n$ and $\C$ be a $k$-dimensional vector MRD code in $\F_{q^m}^n$. Such a code can only exist in the case $m\geq n$ and the $q$-matroid associated to $\C$ is the uniform $q$-matroid $\mathcal{U}_{k,n}(\F_q^n)$ described in \Cref{exp:uniform}, see~\cite[Example 2.4.]{gj2023}. 
    In other words, the uniform $q$-matroid $\mathcal{U}_{k,n}(\F_q^n)$ is representable over $\F_{q^m}^n$ if and only if $m\geq n$.
\end{example}

\section{A lower bound for the number of $q$-matroids}
\label{lower_bound}

In this section, we give a lower bound on the number of $q$-matroids in a fixed dimension. For this purpose, we use the so-called subspace codes and relate them with the paving $q$-matroid construction presented in \Cref{q_mats}.

We start with a brief overview of subspaces codes, see \cite{skk2008,hk2017} for more details.
We consider the set $\cL(\F_q^n)$ endowed with the \textbf{subspace distance} $d_{\text{S}}$, defined as
\[
	d_{\text{S}}(V,W):=\dim(V)+\dim(W)-2\dim(V\cap W)\quad \text{for all } V,W\in\cL(\F_q^n).
\]
Then the pair $(\cL(\F_q^n),d_{\text{S}})$ is a metric space.

\begin{definition}
    A non-empty subset $\C\subseteq\cL(\F_q^n)$ is called a \textbf{subspace code}. The \textbf{minimum subspace distance} is given by
    $$d_\text{S}(\C):=\min\{d_\text{S}(V,W)\;|\;V,W\in\C,V\not=W\}.$$
    Let $\C$ be a subspace code. If the dimensions of all elements of $\C$ are equal, it is called a \textbf{constant dimension code (CDC)}.
    We denote by $A_q(n,d;k)$ the maximal cardinality of a constant dimension code $\C\subseteq{\F_q^n\choose k}_q$ with minimal subspace distance $d_\text{S}(\C)\geq d$.
\end{definition}

One part of our result relies on the following lower bound on the maximal cardinality of constant dimension codes.
\begin{proposition}\label{lower_bound_max_card_CDC}
For $2k\leq n$ and $d\geq 4$, it holds that
\[
    q^{(n-k)\cdot(k-\frac{d}{2}+1)}\leq A_q(n,d;k).
\]
\end{proposition}

This inequality is a consequence of \cite[Section IV]{skk2008}.
The authors essentially use a specific lifting construction to obtain constant dimension codes from MRD codes in $\F_q^{k\times (n-k)}$.
This is achieved by putting an identity matrix as prefix in front of the codewords of the MRD code and then taking the row space of the resulting matrices.
The resulting constant dimension code inherits the distance properties of its underlying matrix rank-metric code. Moreover it also preserves the cardinality of the underlying matrix rank-metric code, which finally yields the lower bound on $A_q(n,d;k)$.    

Next, we describe a connection between CDC's and paving $q$-matroids introduced in \Cref{q_mats}.

\begin{lemma} \label{connection_paving_cdc}
	Let $\mathcal{S}\subseteq\cL(\F_q^n)$ be a CDC of dimension $k$ and minimal distance at least $4$.
	Then $\mathcal{S}$ fulfills the assumption of \Cref{paving_construction} and therefore yields a paving $q$-matroid.
\end{lemma}

\begin{proof}
	Let $V,W\in \mathcal{S}$ be two distinct subspaces.
	By construction we thus have
	\[
	d_{\mathcal{S}}(V,W)=\dim(V) +\dim(W)-2\dim(V\cap W)\ge 4.
	\]
	Since $V,W$ are both $k$-dimensional subspaces, this yields
	$
	\dim(V\cap W)\le k-2
	$
	as desired.
\end{proof}

Let us mention here another connection between paving $q$-matroids and graph theory.

\begin{remark}\label{grassmann_graph_remark}
    Consider the simple graph $J_q(n,k)$, with vertex set given by $\mathcal{V}={\F_q^n\choose k}_q$ and the adjacency relation for two vertices $V,W\in\mathcal{V}$ defined via the condition $\dim(V\cap W)=k-1$. This graph is called the \textbf{Grassmann graph} or \textbf{$q$-Johnson graph}, see \cite{hk2017}.
    
    An \textbf{independent set} $I$ of $J_q(n,k)$ is a subset of the vertex set such that no two vertices in this set are connected by an edge, i.e., a subset $I\subseteq\mathcal{V}$ such that for every two vertices $V,W\in I$ the inequality $\dim(V\cap W)\leq k-2$ is satisfied.
    
    Thus, the independent sets of $J_q(n,k)$ yield paving $q$-matroids via \Cref{paving_construction}. 
  	Note that this is the $q$-analogue of the well-known statement for usual matroids, namely that the independent sets of the Johnson graph are in one-to-one correspondence with sparse paving matroids.
  	This fact is a crucial ingredient in several asymptotic results on matroids, e.g., Knuth's first lower bound on the number of matroids~\cite{Knuth}.
  	We believe that this $q$-analogue observation could also turn out to be useful in futures research on the asymptotic of $q$-matroids.
\end{remark}

Now we prove a lower bound on the number $\mathcal{N}_{q}(k,n)$ of $q$-matroids of rank $k$ in a fixed dimension $n\geq 1$, and subsequently give the lower bound on the number of all $q$-matroids, which will be an immediate consequence of \Cref{lower_bound_q_mats}. Recall that we can restrict ourselves to the case $0\leq k\leq \lfloor\frac{n}{2}\rfloor$ by $q$-matroid duality.

\begin{theorem}
    Let $n\geq 4$ be an integer and $2\leq k\leq \lfloor\frac{n}{2}\rfloor$. Then the number $\mathcal{N}_{q}(k,n)$ of $q$-matroids of rank $k$ on $\F_q^n$ satisfies
    \[
        2^{q^{(n-k)\cdot(k-1)}}< \mathcal{N}_{q}(k,n).
    \]
    \label{lower_bound_q_mats}
\end{theorem}

\begin{proof}
    Let $\mathcal{S}$ be a CDC of maximal size, having minimal subspace distance $d_\text{S}(\mathcal{S})=4$. Since $2 k\leq n$ all the conditions of the estimation in \Cref{lower_bound_max_card_CDC} are fulfilled and thus
    \[
        Q:=q^{(n-k)\cdot(k-4/2+1)}\leq |\mathcal{S}|.
    \]
    Moreover by \Cref{connection_paving_cdc} the set $\mathcal{S}$ forms a paving $q$-matroid in the sense of \Cref{paving_construction}. Then all its subsets also satisfy the intersection-dimension condition of \Cref{paving_construction}, therefore all of them form paving $q$-matroids as well.
    Furthermore all of these $q$-matroids are indeed pairwise distinct since their circuits of rank $k-1$ are exactly the subspaces specified by the subset of $\mathcal{S}$.
    This implies that we have at least $2^Q$-many paving $q$-matroids of rank $k$ on $\F_q^n$, which yields the desired inequality and completes the proof.
    \end{proof}

\begin{corollary}
    Let $n\geq 4$ and let $\mathcal{N}_{q}(n)$ denote the number of all $q$-matroids on $E=\F_q^n$. Then
    \[
        2^{q^{(n-\lfloor\frac{n}{2}\rfloor)\cdot(\lfloor\frac{n}{2}\rfloor-1)}}< \mathcal{N}_{q}(n).
    \]
    \label{corrollary_lower_bound_q_mats}
\end{corollary}

\begin{proof}
    This is just \Cref{lower_bound_q_mats} applied to the case $k=\lfloor\frac{n}{2}\rfloor$.  
\end{proof}

Let us note here that the above bound is not sharp and one way to improve it would be to sum up all bounds for $2\leq k\leq \lfloor\frac{n}{2}\rfloor$.  

\section{An upper bound for the number of representable $q$-matroids}
\label{upper_bound}

\subsection{Zero patterns and their connection to representable $q$-matroids}

In this subsection, we start by giving a brief introduction to the theory of zero patterns, see \cite{rbg2000} for more details. Afterwards, we explain how these patterns are related to representable $q$-matroids. 

\begin{definition}
    We call a string of length $m$ over the alphabet $\{0,*\}$ a \textbf{zero pattern} of length $m$. Let $\K$ be field and $a\in\K$, then set
    \[
        \delta(a):=\begin{cases}
            \;\hfil0\;&\text{if } a=0,\\
            \;\hfil*\;&\text{otherwise}.
        \end{cases}
    \]
    When $a=(a_1,\ldots,a_s)\in\K^s$ we apply $\delta$ coordinate-wise, i.e., $\delta(a)=(\delta(a_1),\ldots,\delta(a_s))$ which is called the \textbf{zero pattern of $a$}.
    
    Lastly let $f=(f_1,\ldots,f_m)$ be a sequence of functions $f_i:D\rightarrow\K$ on a common domain set $D$. Now for an $a\in D$ we call $\delta(f,a):=(\delta(f_1(a)),\ldots,\delta(f_m(a)))$ a \textbf{zero pattern of $f$}.
\end{definition}

For our purposes, we only consider the case that $D=\K^s$ and the $f_i$ are polynomials in $K[x_1,\ldots,x_s]$. Moreover, let us denote the number of all zero patterns of $f$ as $a$ ranges over $D$ by $Z_\K(f)$.

\begin{example}\label{exp: zero_pattern}
    Let $\K=\overline{\F_2}$ and consider the sequence of polynomials in $\overline{\F_2}[x_1,x_2,x_3]$ given by:
    \[
        f=(x_1,x_2,x_3,x_1+x_2,x_1+x_3,x_2+x_3,x_1+x_2+x_3).
    \]
    The element $a=(0,0,1)\in\overline{\F_2}^3$ yields the zero pattern $\delta(f,a)=(0,0,*,0,*,*,*)\in Z_\K(f)$.
\end{example}

We now state a result concerning an upper bound on the number $Z_\K(f)$, which we use later to bound the number of representable $q$-matroids for fixed rank. The proof of the following Theorem is a direct consequence of \cite[Theorem 1.3.]{rbg2000}.

\begin{theorem}[{\cite{rbg2000}}]
    \label{zero_patterns_est}
    Let $f=(f_1,\ldots,f_m)$ be sequence of polynomials in $s$ variables over the field~$\K$, all having degree less or equal $d$ and assume $m\geq s$. Then for $d=1$ we have  $Z_\K(f)\leq\sum_{j=0}^s{m\choose j}$ and for $d\geq 2$ it holds that
    \[
        Z_\K(f)\leq{md\choose s}.
    \]
\end{theorem}

\begin{example}[Continuation of~\Cref{exp: zero_pattern}]
    Continuing~\Cref{exp: zero_pattern}, the sequence $f$ has the parameters $m=7$, $s=3$ and $d=1$. Applying \Cref{zero_patterns_est} yields
    \[
        Z_{\overline{\F_2}}(f)\leq\sum_{j=0}^3{7\choose j}=64.
    \]
    Hence, there are at most $64$ zero patterns of the sequence $f$ over $\overline{\F_2}$.
    The above theorem does not depend on the field, so this upper bound holds in fact over an arbitrary field $\K$.
\end{example}

Next, we want to relate these zero patterns to representable $q$-matroids. To this end, we define the following sequence of polynomials.

\begin{definition}\label{Construction_polynomials}
    Let $n\geq 1$ be an integer and $1\leq k\leq n$. Moreover let us consider the vector space $\F_q^n$ and fix a total ordering on the set of all $k$-dimensional subspaces of $\F_q^n$, i.e., $U_1,\ldots,U_{{n \choose k}_q}$. We define the $(k\times n)$-matrix
    \[
        G=\begin{pmatrix}
            x_{1,1}&\cdots&x_{1,n}\\
            \vdots&\ddots&\vdots\\
            x_{k,1}&\cdots&x_{k,n}
        \end{pmatrix},
    \]
    where the $x_{i,j}$'s are the indeterminates of a polynomial ring $P$ over the algebraic closure of $\F_q$, i.e., $P=\overline{{\F_q}}[x_{i,j}\;|\;1\leq i \leq k, 1\leq j\leq n]$.
    Each $k$-dimensional space $U_i\in{\F_q^n \choose k}_q$ can be regarded as the row space of a matrix $Y_{U_i}\in\F_q^{k\times n}$, i.e., $\text{rowspan}_{\F_q}(Y_{U_i})=U_i$ for all $i=1,\ldots,{n \choose k}_q$.
    Now we define ${n \choose k}_q$-many homogeneous polynomials of degree $k$ in $P$, via
    \[
        f_{U_i}(\boldsymbol{x})=\det(G\cdot Y_{U_i}^T)\quad \text{for all } i=1,\ldots,{n \choose k}_q,
    \]
    where we set $\boldsymbol{x}=(x_{i,j})_{1\leq i\leq k,1\leq j\leq n}$.\ Finally denote by $\mathcal{F}^q_{n,k}$ the sequence of the above polynomials, i.e., $\mathcal{F}^q_{n,k}:=(f_{U_i})_{1\leq i\leq{n \choose k}_q}$.
\end{definition}

The following lemma provides a characterization of representable $q$-matroids in terms of zero patterns of the sequence of polynomials $\mathcal{F}^q_{n,k}$. Let us denote the set of all zero patterns of $\mathcal{F}^q_{n,k}$ by $P_{\overline{{\F_q}}}(\mathcal{F}_{n,k})$.

\begin{lemma}
\label{zero_patterns_rep_q_mat_}
    Let $\M$ be a $q$-matroid of rank $k$ on $\F_q^n$ and $\B$, $\mathcal{NB}$ its collection of bases and non-bases, respectively. Then $\M$ is representable if and only if there exists a zero pattern of $\mathcal{F}^q_{n,k}$ for some $u\in\overline{{\F_q}}^{kn}$ of the form
    \[
        \delta(\mathcal{F}^q_{n,k},u)=(f_{U_i}(u))_{1\leq i\leq{n \choose k}_q}=
        \begin{cases}
            \;\hfil0\;&\emph{if } U_i\in\mathcal{NB},\\
            \;\hfil*\;&\emph{if } U_i\in\B.
        \end{cases}
    \]
\end{lemma}

\begin{proof}
    On the one hand, if $\M$ is representable, then there exists a $(k\times n)$-matrix $G$ in $\F_{q^m}^{k\times n}$, for some $m\geq 1$, representing $\M$, and the entries of $G$ from a vector $u\in\overline{{\F_q}}^{kn}$ such that all $f_{U_i}$ corresponding to non-bases vanish and those corresponding to bases do not. Thus we get the above described zero pattern in $P_{\overline{{\F_q}}}(\mathcal{F}^q_{n,k})$, which concludes the proof of the first direction.
    
    On the other hand, given such a zero pattern $\delta(\mathcal{F}^q_{n,k},u)$ for some $u\in\overline{{\F_q}}^{kn}$, the entries of $u$ form a full-rank $(k\times n)$-matrix $G$ in some $\F_{q^m}^{k\times n}$. Therefore, $\M$ is represented by $G$ and thus representable, which concludes the proof of the second direction.   
\end{proof}

\begin{example}
    Consider the vector space $\F_2^3$ and totally order its $1$-dimensional subspaces by
    \[
        (1 0 0)<(0 1 0)<(0 0 1)<(1 1 0)<(1 0 1)<(0 1 1)<(1 1 1).
    \]
    Using the matrix $G=\begin{pmatrix} x_1&x_2&x_3\end{pmatrix}$ over $\overline{\F_2}[x_1,x_2,x_3]$, the polynomial sequence $f$ from \Cref{exp: zero_pattern} agrees with the sequence $\mathcal{F}^2_{3,1}$ from the construction in \Cref{Construction_polynomials}. Furthermore let $\M$ be the $q$-matroid on $\F_2^3$ with the non-bases $\mathcal{NB}=\{(100),(010),(110)\}$ and bases $\B={\F_2^3 \choose 1}_2\setminus\mathcal{NB}$. According to \Cref{zero_patterns_rep_q_mat_} the $q$-matroid $\M$ is representable if and only if there exist $u\in\overline{\F_2}^3$ such that $\delta(\mathcal{F}^2_{3,1},u)=(0,0,*,0,*,*,*)$. Indeed setting $u=a=(0,0,1)$ as in \Cref{exp: zero_pattern} yields the above zero pattern which  therefore proves that $\M$ is representable.
\end{example}

Now denote by $\mathcal{R}_{q}(k,n)$ the number of representable $q$-matroids on $\F_q^n$ of rank $k$ in a fixed dimension $n\geq 1$. Then the following inequality is a direct consequence of \Cref{zero_patterns_rep_q_mat_}.

\begin{corollary}
    \label{num_rep_q_mats_num_zero_patterns_ineq.}
    Let $0\leq k\leq n$. Then 
    \[
      \mathcal{R}_{q}(k,n)<Z_{\overline{{\F_q}}}(\mathcal{F}^q_{n,k}).
    \]
\end{corollary}

\begin{proof}
    By \Cref{zero_patterns_rep_q_mat_} every representable $q$-matroid has an associated zero pattern in $P_{\overline{{\F_q}}}(\mathcal{F}^q_{n,k})$. But there are zero patterns in $P_{\overline{{\F_q}}}(\mathcal{F}^q_{n,k})$ which do not yield $q$-matroids of rank $k$, for instance $(0,\ldots,0)$ if $k\not=0$. Thus, we obtain the strict inequality stated above.
\end{proof}

\begin{remark}
    The above lemma can be used to algorithmically test, whether a given $q$-matroid is representable over some field using Gr\"obner bases.
    This parallels the well-known algorithm to test matroid representability over algebraically closed fields for instance described in~\cite[Theorem~6.8.14]{o2011}.
\end{remark}

\subsection{The upper bound and the proof of the main result}

In \Cref{lower_bound} we established a lower bound on the number of all $q$-matroids. In this subsection, we turn to the discussion of an upper bound on the number of representable $q$-matroids. This provides us with the last piece for the discussion about the asymptotic behavior of the representable $q$-matroids. We conclude the subsection with the proof of our main result, \Cref{q_mat_asymptotics}.

As in \Cref{lower_bound} we first prove an upper bound on the number $\mathcal{R}_{q}(k,n)$ for fixed $k$ and $n$ and afterward give the upper bound on the number of all representable $q$-matroids, which will be an immediate consequence of \Cref{upper_bound_rep_q_mats}.

Recall again that we restrict ourselves to the case $0\leq k\leq\lfloor\frac{n}{2}\rfloor$ by $q$-matroid duality.

\begin{theorem}
    Let $n\geq 2$ and $1\leq k\leq\lfloor\frac{n}{2}\rfloor$. Then the number of representable $q$-matroids of rank $k$ satisfies the following upper bounds, depending on the rank $k$.
    \begin{enumerate}
        \item If $k=1$ it holds that
        \[
            \mathcal{R}_{q}(1,n)<q^{\log_q(n)+n^2+n\log_q(e)}+1.
        \]

        \item If $2\leq k\leq \lfloor\frac{n}{2}\rfloor$ it holds that
        \[
            \mathcal{R}_{q}(k,n)< \Big(\frac{111}{32}\Big)^{kn}q^{k^2n^2-k^3n+k\log_q(e)n}=q^{n^2k^2-nk^3+nk\log_q(e)+nk\log_q(\frac{111}{32})}.
        \]

        \item Moreover there exists a bound independent of $k$, which holds for every $k\in\{2,\ldots,\lfloor\frac{n}{2}\rfloor\}$. This bound is given by
        \[
            \mathcal{R}_{q}(k,n)< \Big(\frac{111}{32}\Big)^\frac{n^2}{2}q^{\frac{n^2}{4}+\log_q(e)\frac{n^2}{2}}=q^{\frac{n^2}{4}+\log_q(e)\frac{n^2}{2}+\log_q(\frac{111}{32})\frac{n^2}{2}}.
        \]
    \end{enumerate}
    \label{upper_bound_rep_q_mats}
\end{theorem}

We want to give a short remark concerning Part 1 of Theorem \ref{upper_bound_rep_q_mats}.

\begin{remark}
    We know that all different $q$-matroids of rank $1$ on $\F_q^n$ are determined by selecting the dimension of their loop spaces ranging from $0$ to $n-1$. Moreover, all these $q$-matroids are representable. The reason for this is as follows. Firstly it suffices to consider the rank $1$ $q$-matroid $\M^{(1)}_k$ having the $k$-dimensional loop space
    \[
    L_k :=\begin{cases}
    	\;\hfil\langle e_1,\ldots,e_k\rangle\;&\text{if } k\not=0,\\
    	\;\hfil\{0\}\;&\text{otherwise}.
    \end{cases}
    \]
    For any other rank $1$ $q$-matroid $\M$ with a $k$-dimensional loop space $L$ there exists an $\F_q$-isomorphism of $\F_q^n$ mapping $L_k$ to $L$.
    Since this isomorphism thus also maps all bases of $\M^{(1)}_k$ to the ones of $\M$ we regard these two $q$-matroids as isomorphic, in particular the image of a representation matrix of $\M^{(1)}_k$ yields a representation matrix of $\M$.
    Hence without loss of generality it suffices to argue that $\M^{(1)}_k$ is representable.
    
	Secondly, we now construct a representing matrix $G$ of $\M^{(1)}_k$ as
    \[
        \begin{cases}
            \;\hfil G=\begin{pmatrix}0,\ldots,0,1,\alpha,\ldots,\alpha^{n-k-1}\end{pmatrix}\in\F_{q^k}^{1\times n}\;&\text{if } k\not=0 \text{ and}\\
            \;\hfil G=\begin{pmatrix}1,\beta,\ldots,\beta^{n-1}\end{pmatrix}\in\F_{q^n}^{1\times n}\;&\text{otherwise},
        \end{cases}
    \]
    where $\alpha,\beta$ are primitive elements of the field extensions $\F_{q^k}/\F_q$ and $\F_{q^n}/\F_q$, respectively.
    Therefore we can compute the exact number $\mathcal{R}_{q}(1,n)$ by the formula
    \[
        \mathcal{R}_{q}(1,n)=\sum_{i=0}^{n-1}{n\choose i}_q. 
    \]
    This formula arises by utilizing that we only have loop space dimensions $0\leq k\leq n-1$ and that every subspace of $\F_q^n$ of such dimension $k$ provides one. 
    Thus the estimate in \Cref{upper_bound_rep_q_mats} (1) is strictly larger than this sum.
    However, this upper bound is more than enough for our purposes.
\end{remark}

To prove Theorem \ref{upper_bound_rep_q_mats} we first need the following lemma.

\begin{lemma}
    \label{q_binomial_inequality}
    Let $n\geq 2$ and $1\leq k\leq\lfloor\frac{n}{2}\rfloor$. Then the following inequality holds
    \[
        nk\leq{n\choose k}_q.
    \]
\end{lemma}

\begin{proof}
    By \Cref{Estimations_q_binom} we know that ${n\choose k}_q$ can be lower bounded by $q^{(n-k)k}$. Using that one computes
    \[
        {n\choose k}_q\geq q^{(n-k)k}\geq q^{\lceil\frac{n}{2}\rceil k}\geq q^{\lceil\frac{n}{2}\rceil+k}\geq q^{\lceil\frac{n}{2}\rceil}k\geq nk,
    \]
    where for the last inequality we used that $q^{\lceil\frac{n}{2}\rceil}\geq n$ for all $n\geq 2$, which finishes the proof.
\end{proof}

Now we are ready to prove \Cref{upper_bound_rep_q_mats}, using the theory of zero patterns of polynomials.
\begin{proof}[Proof of \Cref{upper_bound_rep_q_mats}]
    Our goal is to apply the inequalities of \Cref{zero_patterns_est} to the sequence of polynomials $\mathcal{F}_{n,k}$ from \Cref{Construction_polynomials} as we know from \Cref{num_rep_q_mats_num_zero_patterns_ineq.} that $\mathcal{R}_{q}(k,n)<Z_{\overline{{\F_q}}}(\mathcal{F}_{n,k})$ holds. 
    \Cref{q_binomial_inequality} ensures that the conditions specified in \Cref{zero_patterns_est} hold for the parameters $m={n\choose k}_q$, $s=kn$, and $d=k$.
    
    For the first statement, keeping in mind that $k=1$, we use the first inequality in \Cref{zero_patterns_est}, which then gives us
    \[
    \mathcal{R}_{q}(1,n)< Z_{\overline{{\F_q}}}(\mathcal{F}_{n,k})\leq\sum_{i=0}^n{{n\choose k}_q\choose i}=1+\sum_{i=1}^n{{n\choose 1}_q\choose i}.
    \]
    The sum can be upper bounded by
    \begin{align*}
        \sum_{i=1}^n{{n\choose 1}_q\choose i}&<\sum_{i=1}^n\Bigg({\frac{q^ne}{i}}\Bigg)^i\\
        & < n(q^ne)^n\\
        &= nq^{n^2}e^n\\
        &= q^{\log_q(n)+n^2+n\log_q(e)},
    \end{align*}
    where the first inequality follows from the general estimate ${n\choose 1}_q\leq q^{n}$, as well as the general bound ${n\choose k}<(\frac{ne}{k})^k$ for the binomial coefficient. In total, this proves the first statement.
    
    Now to prove the second statement of \Cref{upper_bound_rep_q_mats} we consider $2\leq k\leq \lfloor\frac{n}{2}\rfloor$ and therefore  we use the second inequality from \Cref{zero_patterns_est}, which then yields 
    \[
        \mathcal{R}_{q}(k,n)< Z_{\overline{{\F_q}}}(f)\leq{{{n\choose k}_q k}\choose kn}.
    \]
    This binomial coefficient can be upper bounded by
    \begin{align*}
        {{{n\choose k}_q k}\choose kn}&<\Bigg({\frac{\frac{111}{32}q^{(n-k)k}e}{n}}\Bigg)^{kn}\\
        & <\Bigg(\frac{111}{32}q^{(n-k)k}e\Bigg)^{kn}\\
        &= \Big(\frac{111}{32}\Big)^{kn}q^{n^2k^2-nk^3}e^{kn}\\
        &=\Big(\frac{111}{32}\Big)^{kn}q^{n^2k^2-nk^3+nk\log_q(e)}\\
        &= q^{n^2k^2-nk^3+nk\log_q(e)+nk\log_q(111/32)},
    \end{align*}
    where the first inequality follows from the bound for the $q$-binomial coefficient from \Cref{Estimations_q_binom} and the general bound on the binomial coefficient from the previous case. This however proves the second statement. For the third statement, one can use the fact that $n^2k^2-nk^3+nk\log_q(e)+nk\log_2(111/32)$ is strictly increasing as a polynomial in $k$ over the interval $2\leq k\leq \frac{n}{2}$. Thus it attains its maximum in $k=\frac{n}{2}$ and so does $q^{n^2k^2-nk^3+nk\log_q(e)+nk\log_2(111/32)}$, which completes the proof of the third statement of \Cref{upper_bound_rep_q_mats}.  
\end{proof}

The desired upper bound on the number of representable $q$-matroids in $\F_q^n$, is now a direct consequence of Theorem \ref{upper_bound_rep_q_mats}.

\begin{corollary}
    Let $n\geq 2$ be an integer and let $\mathcal{R}_{q}(n)$ be the number of all representable $q$-matroids on $E=\F_q^n$. Then we have
    \begin{align*}
        \mathcal{R}_{q}(n)&< 2\Big(\frac{n}{2}\cdot q^{\frac{n^2}{4}+\log_q(e)\frac{n^2}{2}+\log_q(\frac{111}{32})\frac{n^2}{2}}+q^{\log_q(n)+n^2+n\log_q(e)}+2\Big) \notag\\
        &= 2\Big(q^{\frac{n^2}{4}+\log_q(e)\frac{n^2}{2}+\log_q(\frac{111}{32})\frac{n^2}{2}+\log_q(\frac{n}{2})}+q^{\log_q(n)+n^2+n\log_q(e)}+2\Big).
    \end{align*}
    \label{corollary_upper_bound_rep_q_mats}
\end{corollary}

\begin{proof}
    By \Cref{Duality_rep_q_mats} we know that $\mathcal{R}_{q}(k,n)=\mathcal{R}_{q}(n-k,n)$, therefore we only need to consider $0\leq k\leq \lfloor\frac{n}{2}\rfloor$ and double our estimates.

    The case $k=0$ yields exactly one representable $q$-matroid namely the trivial one, i.e., $\mathcal{U}_{0,n}(E)$. For the case $k=1$ we use \Cref{upper_bound_rep_q_mats} (1) and for each of the cases $2\leq k\leq \lfloor\frac{n}{2}\rfloor$ we use \Cref{upper_bound_rep_q_mats} (3). This provides us a factor of $\lfloor\frac{n}{2}\rfloor-1$, that we bound by $\frac{n}{2}$. Altogether this yields the desired inequality and completes the proof.
\end{proof}

With all preparations done, we are now finally able to prove \Cref{q_mat_asymptotics}.

\begin{proof}[Proof of \Cref{q_mat_asymptotics}]
Consider the ratio $\frac{\mathcal{R}_{q}(n)}{\mathcal{N}_{q}(n)}$.
\Cref{corrollary_lower_bound_q_mats} implies that the denominator of this fraction grows at least doubly exponentially while the numerator grows at most exponentially by~\Cref{corollary_upper_bound_rep_q_mats} when $n$ grows.
Thus in the limit this yields
\[
    \lim_{n\rightarrow\infty}\frac{\mathcal{R}_{q}(n)}{\mathcal{N}_{q}(n)}= 0.\qedhere
\]
\end{proof}

\printbibliography

\end{document}